\documentclass[12pt]{amsproc}

\usepackage{amssymb,amsmath}

\newtheorem{theorem}{Theorem}[section]

\theoremstyle{definition}

\newtheorem{corollary}[theorem]{Corollary}

\theoremstyle{remark}

\numberwithin{equation}{section}

\begin{document}
\title{Minimal Free Resolutions  of the Tangent Cones for Gorenstein Monomial Curves}

\author{P{\i}nar METE}
\address{Department of Mathematics,
Bal{\i}kesir University, Bal{\i}kesir , 10145 Turkey}
\email{pinarm@balikesir.edu.tr}
                \author{Esra Emine ZENG\.{I}N}
\address{Department of Mathematics,
Bal{\i}kesir University, Bal{\i}kesir , 10145 Turkey}
\email{esrazengin103@gmail.com}

\date{\today}

\begin{abstract}
We study the minimal free resolution of the tangent cone of Gorenstein
monomial curves in affine 4-space. We give the explicit minimal free resolution of the tangent cone of non-complete intersection Gorenstein monomial curve whose tangent cone has 
five minimal generators  and show that the possible Betti sequences are  $(1,5,6,2)$ and $(1,5,5,1)$. Also, we compute the Hilbert function of the tangent cone of these families as a result.
\end{abstract}

\maketitle

\section{Introduction}
The minimal free resolution is a central topic in commutative algebra and is a very useful tool  for extracting information about modules. Many algebraic invariants of the module such as Hilbert function, Betti numbers, etc., can be deduced from its minimal free resolution. When the module is associated to a geometric object, these invariants give useful geometric information about it.
Since it is possible to calculate the Hilbert function in terms of the graded Betti numbers, free resolutions play an important role in the theory of Hilbert series. Although the Hilbert function of a standard graded algebra over a field $k$ is well known in the Cohen-Macaulay case, in general, very little is known in local algebra.  The problem which is due to M.E. Rossi  \cite{rossi} asks whether the Hilbert function of a Gorenstein local ring of dimension one is non-decreasing. Recently, it has been shown that there are many families of monomial curves giving negative answer to this problem \cite{oneto-strazzanti-tamone}. But it is still open for Gorenstein local rings associated to monomial curves in affine $d-$space for $3 < d < 10$ and our main aim is to understand the Hilbert function when $d=4$.

Let  $R$ be the polynomial ring $k[x_1, \ldots, x_d]$ over  an arbitrary field $k$.
A monomial affine curve $C$ has a parametrization
\begin{equation} \label{generalmonomial}
x_1=t^{n_1}, \; x_2=t^{n_2}, \; \ldots, \; x_d=t^{n_d}
\end{equation}
where $n_1,n_2 ,\ldots ,n_d$ are positive integers with ${\rm gcd}(n_1,n_2,...,n_d)=1$ and $n_1,n_2 ,\ldots ,n_d$ is a minimal set of generators for the numerical semigroup
\begin{center}
$S=<n_{1},n_2,...,n_{d}>=\{n \mid n=\sum_{i=1}^{d}a_{i}n_{i}, \; a_{i}$'s are
non-negative integers\}. 
\end{center}

The semigroup ring $k[t^{n_1},\ldots,t^{n_d}]$ of $S$ is isomorphic to the coordinate ring $k[x_1,\ldots,x_d]/I(C)$ and the coordinate ring $G=gr_{m}(k[[t^{n_1},\ldots,t^{n_d}]])$ of the tangent cone of a monomial curve $C$ at the origin is isomorphic to the ring $k[x_1,\ldots,x_d]/I(C)_*$. Here, $I(C)_*$ is generated by the polynomials $f_*$, the homogeneous summand of $f$ of the least degree, for $f$ in $I(C)$ where $I(C)$ is the defining ideal of $C$.
 A monomial curve given by the parametrization in (\ref{generalmonomial}) is called a Gorenstein monomial curve, if the associated local ring $k[[t^{n_1},t^{n_2},...,t^{n_d}]]$ is Gorenstein. $k[[t^{n_1},t^{n_2},...,t^{n_d}]]$ is Gorenstein if and only if the corresponding numerical semigroup $S=<n_{1},n_2,...,n_{d}>$ is symmetric \cite{kunz}.

Let $S=<n_{1},n_2,n_3,n_{4}>$ be a 4-generated  numerical semigroup. If $S$ is symmetric and complete intersection, then the Betti sequence of the corresponding semigroup ring is $(1,3,3,1)$, see \cite{stamate}.
Barucci, Fr\"{o}berg and \c{S}ahin  \cite{barucci-froberg-sahin} described
the minimal free resolution of the semigroup ring of $S$, when $S$ is symmetric and not complete intersection and showed that the Betti sequence is (1,5,5,1). 
When S is 4-generated symmetric and non-complete intersection semigroup, the minimal free resolution and  the list of possible Betti sequences of $G=gr_{m}(k[[t^{n_1},t^{n_2},t^{n_3},t^{n_4}]])$ is still unknown \cite{stamate}. For pseudo-symmetric numerical semigroups, see \cite{sahin-sahin1, sahin-sahin}. 
If  $S$  and its tangent cone have the same Betti sequence, then $S$ is of homogeneous type. For a homogeneous type semigroup, the Betti sequence of its Cohen-Macaulay tangent cone can be obtained from a minimal free resolution of its semigroup ring. For details, see \cite{herzog-rossi-valla}. In this article, we study the minimal free resolution of the tangent cone of Gorenstein non-complete intersection monomial curve $C$ in affine 4-space when  the minimal number of generators of its tangent cone  is  five. Since homogeneous type  semigroups have Cohen-Macaulay tangent cones and the Cohen-Macaulayness of tangent cones of these families of curves was shown in \cite{arslan-mete}, here we consider only 5-generated tangent cones. Based on the Buchsbaum-Eisenbud Theorem \cite{buchsbaum-eisenbud} and knowing the minimal generators of the defining ideal of the tangent cone in four cases \cite{arslan-mete}, we give the minimal free resolution of the tangent cone explicitly. Then, we  compute the Hilbert function of the tangent cone for these families as corollaries. All computations have been done using {\footnotesize SINGULAR} \cite{singular}.

\section{Bresinsky's Theorem}
In \cite{bresinsky}, Bresinsky  gives the explicit description of the defining ideal of a non-complete intersection Gorenstein monomial curve with embedding dimension four by the following theorem.

\begin{theorem}\label{Bresinsky}
Let $C$ be a monomial curve having the parametrization
\[
x_1=t^{n_1}, \; x_2=t^{n_2}, \; x_3=t^{n_3}, \; x_4=t^{n_4}
\]
where $S$ is a numerical semigroup minimally generated by $n_1,n_2,n_3,n_4$.  $S$ is symmetric and $C$ is a non-complete intersection monomial curve if and only if
$I(C)$ is generated by the set
\[
\begin{split}
G=\{ &
f_{1}=x_{1}^{\alpha_{1}}-x_{3}^{\alpha_{13}}x_{4}^{\alpha_{14}},
f_{2}=x_{2}^{\alpha_{2}}-x_{1}^{\alpha_{21}}x_{4}^{\alpha_{24}},
f_{3}=x_{3}^{\alpha_{3}}-x_{1}^{\alpha_{31}}x_{2}^{\alpha_{32}}, \\
& f_{4}=x_{4}^{\alpha_{4}}-x_{2}^{\alpha_{42}}x_{3}^{\alpha_{43}},
f_{5}=x_{3}^{\alpha_{43}}x_{1}^{\alpha_{21}}-x_{2}^{\alpha_{32}}x_{4}^{\alpha_{14}}\}
\end{split}
\]
where the polynomials $f_{i}$'s are unique up to
isomorphism with $0 < \alpha_{ij} < \alpha_{j}$ with $\alpha_{i}n_i \in <n_1,\ldots,\hat{n}_i,\ldots,n_4>$ such that
$\alpha_{i}$'s are minimal for $1\leq i \leq 4$, where $\hat{n}_i$ denotes that
$n_i \notin <n_1,\ldots,\hat{n}_i,\ldots,n_4>$.
\end{theorem}

In Theorem 2.1, there is no restriction on the order of $n_1,\ldots, n_4$ and the set $G$ is valid for only
a permutation of these numbers. If we assume that $n_1<n_2<n_3<n_4$, then we have to revise the set $G$ with respect to the correct permutation of the variables $x_1,x_2,x_3,x_4$. Thus, 
there are six isomorphic possible
permutations which can be considered within three cases:

\begin{enumerate}

\item $f_1=(1,(3,4))$

\begin{enumerate}

\item
$f_2=(2,(1,4)), f_3=(3,(1,2)), f_4=(4,(2,3)), f_5=((1,3),(2,4))$
\item $f_2=(2,(1,3)), f_3=(3,(2,4)), f_4=(4,(1,2)), f_5=((1,4),(2,3))$
\end{enumerate}

\item $f_1=(1,(2,3))$

\begin{enumerate}

\item
$f_2=(2,(3,4)), f_3=(3,(1,4)), f_4=(4,(1,2)), f_5=((2,4),(1,3))$
\item  $f_2=(2,(1,4)), f_3=(3,(2,4)), f_4=(4,(1,3)), f_5=((1,3),(4,2))$
\end{enumerate}

\item $f_1=(1,(2,4))$

\begin{enumerate}

\item $f_2=(2,(1,3)), f_3=(3,(1,4)), f_4=(4,(2,3)), f_5=((1,2),(3,4))$
\item $f_2=(2,(3,4)), f_3=(3,(1,2)), f_4=(4,(1,3)), f_5=((2,3),(1,4))$
\end{enumerate}

\end{enumerate}

\vskip2mm\noindent Here, the notations $f_i=(i,(j,k))$ and $f_5=((i,j),(k,l))$
denote the generators
$f_i=x_i^{\alpha_i}-x_j^{\alpha_{ij}}x_k^{\alpha_{ik}}$ and
$f_5=x_i^{\alpha_{ki}} x_j^{\alpha_{lj}}-x_k^{\alpha_{jk}}x_l^{\alpha_{il}}$
Thus, if we have the extra condition $n_1 < n_2 < n_3 < n_4$, then the generator set of its defining ideal is exactly given by one of these six permutations.

In \cite{arslan-mete}, Arslan and Mete observed that the generator set of each of these curves turned out to be a standard basis with respect to
the negative degree reverse lexicographical ordering in the following cases:
\begin{itemize}
\item In Case 1(a) with the restriction $\alpha_2 \leq
\alpha_{21}+\alpha_{24}$,
\item In Case 1(b) with the restriction $\alpha_2 \leq
\alpha_{21}+\alpha_{23}$, $\alpha_3\leq
\alpha_{32}+\alpha_{34}$.
\item In Case 2(b) with the restriction $\alpha_2
\leq \alpha_{21}+\alpha_{24}$, $\alpha_3 \leq
\alpha_{32}+\alpha_{34}$.
\item In Case 3(a) with the restriction
$\alpha_2 \leq \alpha_{21}+\alpha_{23}$, $\alpha_3 \leq
\alpha_{31}+\alpha_{34}$
\end{itemize}
 \noindent And in all above cases,  the minimal number of generators of the tangent cone of a Gorenstein non-complete intersection monomial curve is five. One can also see \cite{arslan-katsabekis-nalbandiyan}.

\section{Minimal Free Resolutions}

In this section, we give the minimal free resolution of the tangent cone  of Gorenstein non-complete intersection monomial curve $C$ in embedding dimension four when  the minimal number of generators of the tangent cone of $C$ is  five.

\vspace{2mm}\noindent\textbf{Case 1(a) :}
Let
$$f_{1}=x_{1}^{\alpha_{1}}-x_{3}^{\alpha_{13}}x_{4}^{\alpha_{14}},
f_{2}=x_{2}^{\alpha_{2}}-x_{1}^{\alpha_{21}}x_{4}^{\alpha_{24}},
f_{3}=x_{3}^{\alpha_{3}}-x_{1}^{\alpha_{31}}x_{2}^{\alpha_{32}}, \\
f_{4}=x_{4}^{\alpha_{4}}-x_{2}^{\alpha_{42}}x_{3}^{\alpha_{43}}$$ and 
$$f_{5}=x_{3}^{\alpha_{43}}x_{1}^{\alpha_{21}}-x_{2}^{\alpha_{32}}x_{4}^{\alpha_{14}}$$

\noindent where $\alpha_{1}=\alpha_{21}+\alpha_{31}$,  $\alpha_{2}=\alpha_{32}+\alpha_{42}$,
 $\alpha_{3}=\alpha_{13}+\alpha_{43}$,  $\alpha_{4}=\alpha_{14}+\alpha_{24}.$
The condition $n_1<n_2<n_3<n_4$ implies
$\;\alpha_1>\alpha_{13}+\alpha_{14}, \; \alpha_4<\alpha_{42}+\alpha_{43}\;$ and $\;\alpha_3<\alpha_{31}+\alpha_{32}.\;$
Since the extra condition $\alpha_2\leq \alpha_{21}+\alpha_{24}$ and using Lemma 5.5.1 in \cite{greuel-pfister}, the defining ideal $I(C)_*$ of the tangent cone is generated by the following sets: \vspace{0.3cm}

\begin{itemize}
 \item$Case \;1(a1)\;: I(C)_*=(x_{3}^{\alpha_{13}}x_{4}^{\alpha_{14}},x_{2}^{\alpha_{2}},x_{3}^{\alpha_{3}},x_{4}^{\alpha_{4}},x_{2}^{\alpha_{32}}x_{4}^{\alpha_{14}})$
\item$Case\; 1(a2)\;: I(C)_*=(x_{3}^{\alpha_{13}}x_{4}^{\alpha_{14}},x_{2}^{\alpha_{2}}-x_1^{\alpha_{21}}x_{4}^{\alpha_{24}} ,x_{3}^{\alpha_{3}},x_{4}^{\alpha_{4}},x_{2}^{\alpha_{32}}x_{4}^{\alpha_{14}})$
\end{itemize}


\begin{theorem}\label{thm2} In  $Case \;1(a1)$ and $Case \;1(a2)$, the sequence of R-modules

\begin{center}
$ 0 \rightarrow R^2 \xrightarrow{\phi_3} R^6 \xrightarrow{\phi_2} R^5\xrightarrow{\phi_1} R^1  \rightarrow R/I(C)_* \rightarrow 0 $
\end{center}

is a minimal free resolution for the tangent cone of $C$, where
{\footnotesize
$$
\phi_1=\Big(
\begin{matrix}
x_{3}^{\alpha_{13}}x_{4}^{\alpha_{14}} &
x_{2}^{\alpha_{2}} &
x_{3}^{\alpha_{3}} &
x_{4}^{\alpha_{4}}&
x_{2}^{\alpha_{32}}x_{4}^{\alpha_{14}} 
\end{matrix}\Big),
$$\\[-1cm]

$$
\phi_2=
\begin{pmatrix}
x_3^{\alpha_{43}} & 0 & x_4^{\alpha_{24}} & x_2^{\alpha_{32}} & 0 & 0\\[0.5mm]
0 & x_3^{\alpha_{3}} & 0 & 0 & 0 & x_4^{\alpha_{14}}\\[0.5mm]
-x_4^{\alpha_{14}} & -x_2^{\alpha_{2}} & 0 & 0 & 0 & 0 \\[0.5mm]
0 & 0 & -x_3^{\alpha_{13}} & 0 &  x_2^{\alpha_{32}} & 0\\[0.5mm]
0 & 0 & 0  &  -x_3^{\alpha_{13}}  & -x_4^{\alpha_{24}}& -x_2^{\alpha_{42}}\\
\end{pmatrix},
\phi_3=
\begin{pmatrix}
x_{2}^{\alpha_{2}} & 0\\[0.5mm]
-x_{4}^{\alpha_{14}} & 0 \\[0.5mm]
0 & x_{2}^{\alpha_{32}} \\[0.5mm]
 -x_{2}^{\alpha_{42}}x_{3}^{\alpha_{43}} & -x_{4}^{\alpha_{24}}\\[0.5mm]
0 & x_{3}^{\alpha_{13}} \\[0.5mm]
x_{3}^{\alpha_{3}} & 0\\
\end{pmatrix}
$$
}
or 

{\footnotesize
$$
\phi_1=\Big(
\begin{matrix}
x_{3}^{\alpha_{13}}x_{4}^{\alpha_{14}} &
x_{2}^{\alpha_{2}}-x_1^{\alpha_{21}}x_{4}^{\alpha_{24}} &
x_{3}^{\alpha_{3}} &
x_{4}^{\alpha_{4}}&
x_{2}^{\alpha_{32}}x_{4}^{\alpha_{14}} 
\end{matrix}\Big),
$$\\[-5mm]

$$
\phi_2=
\begin{pmatrix}
x_4^{\alpha_{24}} & x_2^{\alpha_{32}} & x_3^{\alpha_{43}} & 0 & 0 & 0\\[0.5mm]
0 & 0 & 0 & x_4^{\alpha_{14}} & 0 & x_3^{\alpha_{3}}\\[0.5mm]
0 &  0 & -x_4^{\alpha_{14}} & 0 & 0 & -x_2^{\alpha_{2}}+x_1^{\alpha_{21}}x_4^{\alpha_{24}}\\[0.5mm]
-x_3^{\alpha_{13}} & 0 & 0 & x_1^{\alpha_{21}} & x_2^{\alpha_{32}} & 0\\[0.5mm]
0 & -x_3^{\alpha_{13}} & 0  &  -x_2^{\alpha_{42}}  & -x_4^{\alpha_{24}} & 0\\
\end{pmatrix},
\phi_3=
\begin{pmatrix}
x_{2}^{\alpha_{32}} & x_{1}^{\alpha_{21}}x_{3}^{\alpha_{43}}\\[0.5mm]
-x_{4}^{\alpha_{24}} & -x_{2}^{\alpha_{42}}x_{3}^{\alpha_{43}} \\[0.5mm]
0 & x_{2}^{\alpha_{2}}-x_{1}^{\alpha_{21}}x_{4}^{\alpha_{24}} \\[0.5mm]
0 & x_{3}^{\alpha_{3}}\\[0.5mm]
x_{3}^{\alpha_{13}} & 0 \\[0.5mm]
0 & -x_{4}^{\alpha_{14}}\\
\end{pmatrix}
$$
}
respectively.
\end{theorem}

\begin{proof}$Case \;1(a1) :$ 
It is easy to show that $\phi_1\phi_2=\phi_2\phi_3=0$ proving that the sequence above is a complex. To prove the exactness, we
use Buchsbaum-Eisenbud criterion \cite{buchsbaum-eisenbud}. Therefore, first we need to check that 
\begin{center}
$rank(\phi_1)+rank(\phi_2)=1+4=5$ and
$rank(\phi_2)+rank(\phi_3)=4+2=6.$
\end{center}
Clearly, $rank(\phi_1)=1$ and $rank(\phi_3)=2$. Since every $5 \times 5$ minors of $\phi_2$ is zero, by McCoy's Theorem $rank(\phi_2) \leq 4$. In matrix $\phi_2$ , deleting the 1st and the 3rd columns, and the 2nd row, we have $-x_2^{2\alpha_{2}+\alpha_{32}}$ and similarly, deleting the 3rd row, and the 5th and the 6th columns, we obtain 
$x_3^{2\alpha_{3}+\alpha_{13}}$ as $4 \times 4-$minors of $\phi_2$. So, $rank(\phi_2)=4$. These two determinants are relatively prime, so $I(\phi_2)$ contains a regular sequence of length 2.  Among the 2-minors of $\phi_3$, we have $x_2^{\alpha_2+\alpha_{32}}$, $x_3^{\alpha_3+\alpha_{13}}$ and
$x_4^{\alpha_4}$ and this is a  regular sequence, since $\{x_2,x_3,x_4\}$  is a regular sequence. Thus, $I(\phi_3)$ contains a regular sequence of length 3.

$Case \;1(a2) :$ 
Similar to the first case, it is clear that $\phi_1\phi_2=\phi_2\phi_3=0$.  $rank(\phi_1)=1$ and $rank(\phi_3)=2$ are trivial. In matrix $\phi_2$, deleting the 3rd row, and the 4th and the 5th columns, we have $x_3^{2{\alpha_{3}}+{\alpha_{13}}}$ and similarly, deleting the 2nd and the 6th columns, and the 4th row, we obtain $x_4^{2\alpha_{4}}$ and these determinants are relatively prime. 
2-minors of $\phi_3$ are 
\begin{center}
$-x_3^{\alpha_{43}}f_2$, $x_2^{\alpha_{32}}f_2$, $x_2^{\alpha_{32}}x_3^{\alpha_{3}}$, $-x_1^{\alpha_{21}}x_3^{\alpha_{3}}$, $-x_2^{\alpha_{32}}x_4^{\alpha_{14}}$, $-x_4^{\alpha_{24}}f_2$, $-x_4^{\alpha_{24}}x_3^{\alpha_{3}}$, $x_2^{\alpha_{42}}x_3^{\alpha_{3}}$, $x_4^{\alpha_{4}}$, $-x_3^{\alpha_{13}}f_2$, $-x_3^{\alpha_{3}+{\alpha_{13}}}$, $-x_3^{\alpha_{13}}x_4^{\alpha_{14}}.$
\end{center}
\noindent Among these 2-minors of $\phi_3$, we have 
$\{x_2^{\alpha_{32}}f_2$, $-x_3^{\alpha_3+\alpha_{13}}$, $x_4^{\alpha_{4}}\}$. Since $x_2$ is a nonzero divisor modulo $\{x_3,x_4\}$, $I(\phi_3)$ contains a regular sequence of length 3. 
\end{proof}

\vspace{2mm}\noindent\textbf{Case 1(b) :} Let
$$f_{1}=x_{1}^{\alpha_{1}}-x_{3}^{\alpha_{13}}x_{4}^{\alpha_{14}},
f_{2}=x_{2}^{\alpha_{2}}-x_{1}^{\alpha_{21}}x_{3}^{\alpha_{23}},
f_{3}=x_{3}^{\alpha_{3}}-x_{2}^{\alpha_{32}}x_{4}^{\alpha_{34}}, \\
f_{4}=x_{4}^{\alpha_{4}}-x_{1}^{\alpha_{41}}x_{2}^{\alpha_{42}}$$ and 
$$f_{5}=x_{2}^{\alpha_{42}}x_{3}^{\alpha_{13}}-x_{1}^{\alpha_{21}}x_{4}^{\alpha_{34}}$$

\noindent Here, $\alpha_{1}=\alpha_{21}+\alpha_{41}$,  $\alpha_{2}=\alpha_{32}+\alpha_{42}$,
 $\alpha_{3}=\alpha_{13}+\alpha_{23}$,  $\alpha_{4}=\alpha_{14}+\alpha_{34}.$ The condition $n_1<n_2<n_3<n_4$ implies
$\alpha_1>\alpha_{13}+\alpha_{14},$ and $ \alpha_4<\alpha_{41}+\alpha_{42}$.
The extra condition $\alpha_2\leq \alpha_{21}+\alpha_{23}$ and  $\alpha_3\leq \alpha_{32}+\alpha_{34}$
again using Lemma 5.5.1 in \cite{greuel-pfister} imply that the defining ideal $I(C)_*$ of the tangent cone is generated by the following sets: \vspace{0.3cm}

\begin{itemize}
\item$Case \;1(b1)\;: I(C)_*=(x_{3}^{\alpha_{13}}x_{4}^{\alpha_{14}}, x_{2}^{\alpha_{2}}, x_{3}^{\alpha_{3}}, x_{4}^{\alpha_{4}}, 
x_{2}^{\alpha_{42}}x_{3}^{\alpha_{13}})$

\item$Case \;1(b2)\;: I(C)_*=(x_{3}^{\alpha_{13}}x_{4}^{\alpha_{14}},x_{2}^{\alpha_{2}}-x_1^{\alpha_{21}}x_{3}^{\alpha_{23}} ,x_{3}^{\alpha_{3}},x_{4}^{\alpha_{4}},x_{2}^{\alpha_{42}}x_{3}^{\alpha_{13}})$

\item$Case \;1(b3) \;:I(C)_*=(x_{3}^{\alpha_{13}}x_{4}^{\alpha_{14}},x_{2}^{\alpha_{2}},x_{3}^{\alpha_{3}}-x_{2}^{\alpha_{32}}x_{4}^{\alpha_{34}},x_{4}^{\alpha_{4}},x_{2}^{\alpha_{42}}x_{3}^{\alpha_{13}})$

\item$Case \;1(b4)\;: I(C)_*\!=\!(x_{3}^{\alpha_{13}}x_{4}^{\alpha_{14}},x_{2}^{\alpha_{2}}-x_1^{\alpha_{21}}x_{3}^{\alpha_{23}} ,x_{3}^{\alpha_{3}}\!-\!x_{2}^{\alpha_{32}}x_{4}^{\alpha_{34}},x_{4}^{\alpha_{4}},x_{1}^{\alpha_{21}}x_{4}^{\alpha_{34}}\!-\!x_{2}^{\alpha_{42}}x_{3}^{\alpha_{13}})$
\end{itemize}

\begin{theorem}\label{thm3} 
In $Case \;1(b1)$, the minimal free resolution for the tangent cone of C is
\begin{center}
$ 0 \rightarrow R^2 \xrightarrow{\phi_3} R^6 \xrightarrow{\phi_2} R^5\xrightarrow{\phi_1} R^1  \rightarrow R/I(C)_* \rightarrow 0 $
\end{center}
\noindent where
{\footnotesize
$$
\phi_1=
\begin{pmatrix}
x_3^{\alpha_{13}}x_4^{\alpha_{14}} &
x_{2}^{\alpha_{2}}&
x_{3}^{\alpha_{3}} &
x_{4}^{\alpha_{4}}&
x_{2}^{\alpha_{42}}x_{3}^{\alpha_{13}} 
\end{pmatrix},
$$\\[-1cm]

$$
\phi_2=
\begin{pmatrix}
-x_4^{\alpha_{34}} & 0 &  -x_3^{\alpha_{23}}  & 0 & x_2^{\alpha_{42}}  & 0 \\[0.5mm]
0 & -x_4^{\alpha_{4}} & 0 & 0 & 0 & -x_3^{\alpha_{13}}  \\[0.5mm]
0 & 0 & x_4^{\alpha_{14}} & x_2^{\alpha_{42}} & 0  & 0 \\[0.5mm]
 x_3^{\alpha_{13}} & x_2^{\alpha_{2}} & 0 & 0 & 0 & 0\\[0.5mm]
0 & 0 & 0 & -x_3^{\alpha_{23}}&  -x_4^{\alpha_{14}} & x_2^{\alpha_{32}} \\
\end{pmatrix},
\hspace{3mm}
\phi_3=
\begin{pmatrix}
x_{2}^{\alpha_{2}} & 0\\[0.5mm]
-x_{3}^{\alpha_{13}} & 0 \\[0.5mm]
0 & x_{2}^{\alpha_{42}} \\[0.5mm]
0 & -x_{4}^{\alpha_{14}}\\[0.5mm]
 x_{2}^{\alpha_{32}}x_{4}^{\alpha_{34}} & x_{3}^{\alpha_{23}} \\
x_{4}^{\alpha_{4}} & 0
\end{pmatrix},
$$
}

in $Case \;1(b2)$,  the minimal free resolution for the tangent cone of C is

\begin{center}
$ 0 \rightarrow R^2 \xrightarrow{\phi_3} R^6 \xrightarrow{\phi_2} R^5\xrightarrow{\phi_1} R^1  \rightarrow R/I(C)_* \rightarrow 0 $
\end{center}

{\footnotesize
$$
\phi_1=
\begin{pmatrix}
x_3^{\alpha_{13}}x_4^{\alpha_{14}} &
x_{2}^{\alpha_{2}}-x_{1}^{\alpha_{21}}x_{3}^{\alpha_{23}}&
x_{3}^{\alpha_{3}} &
x_{4}^{\alpha_{4}}&
x_{2}^{\alpha_{42}}x_{3}^{\alpha_{13}} 
\end{pmatrix},
$$\\[-5mm]

$$
\phi_2=
\begin{pmatrix}
-x_4^{\alpha_{34}} & 0 &  -x_3^{\alpha_{23}}  & 0 & 0 & x_2^{\alpha_{42}}   \\[0.5mm]
0 & -x_4^{\alpha_{4}} & 0 & 0 &  x_3^{\alpha_{13}} & 0 \\[0.5mm]
0 & 0 & x_4^{\alpha_{14}} & x_2^{\alpha_{42}} & x_{1}^{\alpha_{21}}  & -x_{4}^{\alpha_{14}} \\[0.5mm]
 x_3^{\alpha_{13}} & x_2^{\alpha_{2}}-x_{1}^{\alpha_{21}}x_{3}^{\alpha_{23}} & 0 & 0 & 0 & 0\\[0.5mm]
0 & 0 & 0 & -x_3^{\alpha_{23}}&  -x_2^{\alpha_{32}} & 0\\
\end{pmatrix},
\phi_3=
\begin{pmatrix}
x_{2}^{\alpha_{2}}-x_{1}^{\alpha_{21}}x_{3}^{\alpha_{23}} & 0\\[0.5mm]
-x_{3}^{\alpha_{13}} & 0 \\[0.5mm]
x_{1}^{\alpha_{21}}x_{4}^{\alpha_{34}} & x_{2}^{\alpha_{42}} \\[0.5mm]
0 & -x_{4}^{\alpha_{14}}\\[0.5mm]
-x_{4}^{\alpha_{4}} & 0\\[0.5mm]
 x_{2}^{\alpha_{32}}x_{4}^{\alpha_{34}} & x_{3}^{\alpha_{23}} 
\end{pmatrix},
$$
}

\noindent in $Case \;1(b3)$, 
\begin{center}
$0 \rightarrow R^1 \xrightarrow{\phi_3} R^5\xrightarrow{\phi_2} R^5 \xrightarrow{\phi_1} R^1  \rightarrow G \rightarrow 0$
\end{center}
\noindent where

{\footnotesize
$$
\phi_1=
\begin{pmatrix}
x_3^{\alpha_{13}}x_4^{\alpha_{14}} &
x_{2}^{\alpha_{2}}&
x_{3}^{\alpha_{3}}-x_{2}^{\alpha_{32}}x_{4}^{\alpha_{34}} &
x_{4}^{\alpha_{4}}&
x_{2}^{\alpha_{42}}x_{3}^{\alpha_{13}} 
\end{pmatrix},
$$\\[-1cm]

$$
\phi_2=
\begin{pmatrix}
-x_3^{\alpha_{23}} & 0 &  x_4^{\alpha_{34}}  &  x_2^{\alpha_{42}}  & 0 \\[0.5mm]
0 & x_4^{\alpha_{34}} & 0 & 0 & x_3^{\alpha_{13}}  \\[0.5mm]
x_4^{\alpha_{14}} & x_2^{\alpha_{42}} & 0  & 0 & 0\\[0.5mm]
x_2^{\alpha_{32}} & 0 & -x_3^{\alpha_{13}} & 0 & 0 \\[0.5mm]
0  & -x_3^{\alpha_{23}}& 0 &  -x_4^{\alpha_{14}} & -x_2^{\alpha_{32}} \\
\end{pmatrix},
\hspace{5mm}
\phi_3=
\begin{pmatrix}
x_{2}^{\alpha_{42}}x_{3}^{\alpha_{13}} \\[0.5mm]
-x_{3}^{\alpha_{13}} x_{4}^{\alpha_{14}}  \\[0.5mm]
 x_{2}^{\alpha_{2}} \\[0.5mm]
x_{3}^{\alpha_{3}}-x_{2}^{\alpha_{32}}x_{4}^{\alpha_{34}}\\[0.5mm]
x_{4}^{\alpha_{4}} \\
\end{pmatrix},
$$
}

\noindent and in $Case \;1(b4)$, then the minimal free resolution of the tangent cone of $C$ is
\begin{center}
$0 \rightarrow R^1 \xrightarrow{\phi_3} R^5\xrightarrow{\phi_2} R^5 \xrightarrow{\phi_1} R^1  \rightarrow G \rightarrow 0$
\end{center}
\noindent where

{\footnotesize
$$
\phi_1=
\begin{pmatrix}
x_{3}^{\alpha_{13}}x_{4}^{\alpha_{14}} &
x_{2}^{\alpha_{2}}-x_{1}^{\alpha_{21}}x_{3}^{\alpha_{23}}&
x_{3}^{\alpha_{3}}-x_{2}^{\alpha_{32}}x_{4}^{\alpha_{34}} &
x_{4}^{\alpha_{4}}&
x_{2}^{\alpha_{42}}x_{3}^{\alpha_{13}}-x_{1}^{\alpha_{21}}x_{4}^{\alpha_{34}}
\end{pmatrix},
$$

$$
\phi_2=
\begin{pmatrix}
-x_4^{\alpha_{34}} &  x_3^{\alpha_{23}}  & -x_2^{\alpha_{42}}  & 0 & 0\\[0.5mm]
0 &  0 & 0 & x_4^{\alpha_{34}} &  -x_3^{\alpha_{13}}  \\[0.5mm]
0 &  -x_4^{\alpha_{14}} &0 & x_2^{\alpha_{42}}  &  -x_1^{\alpha_{21}}\\[0.5mm]
x_3^{\alpha_{13}} & -x_2^{\alpha_{32}} & x_1^{\alpha_{21}} & 0 & 0\\[0.5mm]
0 & 0  &   x_4^{\alpha_{14}} & -x_3^{\alpha_{23}} & x_2^{\alpha_{32}} \\
\end{pmatrix},
\hspace{5mm}
\phi_3=
\begin{pmatrix}
x_{2}^{\alpha_{2}}-x_{1}^{\alpha_{21}}x_{3}^{\alpha_{23}}\\[0.5mm]
x_{2}^{\alpha_{42}}x_{3}^{\alpha_{13}}-x_{1}^{\alpha_{21}}x_{4}^{\alpha_{34}} \\[0.5mm]
x_{3}^{\alpha_{3}}-x_{2}^{\alpha_{32}}x_{4}^{\alpha_{34}}\\[0.5mm]
x_{3}^{\alpha_{13}}x_{4}^{\alpha_{14}}\\[0.5mm]
x_{4}^{\alpha_{4}}
\end{pmatrix}.
$$
}
\end{theorem}

\begin{proof} $Case \;1(b1) :$ 
$rank(\phi_1)=1$ and $rank(\phi_3)=2$. As in the $Case 1(a)$,  in matrix $\phi_2$, deleting the 2nd and the 5th columns, and the 3rd row, we have $x_3^{2\alpha_{3}}$ and similarly, deleting the 2nd row, and the 1st and the 3rd columns, we obtain 
$-x_{2}^{2\alpha_{2}+\alpha_{42}}$  as $4 \times 4-$ minors of $\phi_2$. These two determinants are relatively prime, so $I(\phi_2)$ contains a regular sequence of length 2. 
Among the 2-minors of $\phi_3$, we have $x_2^{\alpha_{2}+\alpha_{42}}$, $-x_{3}^{\alpha_3}$ and
$x_4^{\alpha_{4}+\alpha_{14}}$ and these three determinants constitute a regular sequence. Thus, $I(\phi_3)$ contains a regular sequence of length 3.\\

$Case \;1(b2) :$ 
It is clear that  $rank(\phi_1)=1$ and  $rank(\phi_3)=2$. In matrix 
$\phi_2$, deleting the 3rd row, and the 2nd and the 6th columns, we have $-x_3^{2{\alpha_{3}}}$ and deleting the 4th and the 6th columns, and the 4th row, we obtain $-x_2^{\alpha_{32}}x_4^{2\alpha_{4}}$ and these determinants are relatively prime. 
Among the 2-minors of $\phi_3$, we have $-x_3^{\alpha_{3}}$ , $-x_4^{\alpha_{4}+\alpha_{14}}$ and
$x_2^{\alpha_{42}}{f_2}$. Since $x_2$ is a nonzero divisor modulo $\{x_3,x_4\}$, $I(\phi_3)$ contains a regular sequence of length 3.\\

$Case \;1(b3) :$ 
Clearly, $rank(\phi_1)=rank(\phi_3)=1$.
In matrix 
$\phi_2$, deleting the 2nd row and the 3rd column, we have $x_{2}^{2\alpha_{2}}$ and deleting the 4th row and the 5th column, we get $x_{4}^{2\alpha_{4}}$. Since these determinants are powers of different variables, they constitute a regular sequence of length 2.\\

$Case \;1(b4) :$ 
As in the above case, 
$rank(\phi_1)=rank(\phi_3)=1$.
In matrix 
$\phi_2$, deleting 2nd row and 1st column, we have ${f_2}^{2}$ and deleting 4th row  and 5th column, we obtain $x_{4}^{2\alpha_{4}}$. These two determinans are relatively prime, they constitute a regular sequence. $I(\phi_2)$ contains a regular sequence of length 2.
\end{proof}

\vspace{2mm}\noindent\textbf{Case 2(b) :} Let
$$f_{1}=x_{1}^{\alpha_{1}}-x_{2}^{\alpha_{12}}x_{3}^{\alpha_{13}},
f_{2}=x_{2}^{\alpha_{2}}-x_{1}^{\alpha_{21}}x_{4}^{\alpha_{24}},
f_{3}=x_{3}^{\alpha_{3}}-x_{2}^{\alpha_{32}}x_{4}^{\alpha_{34}}, \\
f_{4}=x_{4}^{\alpha_{4}}-x_{1}^{\alpha_{41}}x_{3}^{\alpha_{43}}$$ and 
$$f_{5}=x_{1}^{\alpha_{41}}x_{2}^{\alpha_{32}}-x_{3}^{\alpha_{13}}x_{4}^{\alpha_{24}}.$$

\noindent Here, $\alpha_{1}=\alpha_{21}+\alpha_{41}$,  $\alpha_{2}=\alpha_{12}+\alpha_{32}$,
 $\alpha_{3}=\alpha_{13}+\alpha_{43}$,  $\alpha_{4}=\alpha_{24}+\alpha_{34}.$
The condition $n_1<n_2<n_3<n_4$ implies
$\alpha_1>\alpha_{12}+\alpha_{13}$ and $ \alpha_4<\alpha_{41}+\alpha_{43}$.
Since the extra condition $\alpha_2\leq \alpha_{21}+\alpha_{24}$, $\alpha_3\leq \alpha_{32}+\alpha_{34}$
and Lemma 5.5.1 in \cite{greuel-pfister}, the defining ideal $I(C)_*$ of the tangent cone is generated by the following sets: \vspace{0.2cm}

\begin{itemize}
\item$Case \;2(b1)\;:  I(C)_*=(x_{2}^{\alpha_{12}}x_{3}^{\alpha_{13}}, x_{2}^{\alpha_{2}}, x_{3}^{\alpha_{3}}, x_{4}^{\alpha_{4}}, 
x_{3}^{\alpha_{13}}x_{4}^{\alpha_{24}})$

\item$Case \;2(b2)\;:  I(C)_*=(x_{2}^{\alpha_{12}}x_{3}^{\alpha_{13}},x_{2}^{\alpha_{2}}-x_1^{\alpha_{21}}x_{4}^{\alpha_{24}} ,x_{3}^{\alpha_{3}},x_{4}^{\alpha_{4}},x_{3}^{\alpha_{13}}x_{4}^{\alpha_{24}})$

\item$Case \;2(b3)\;: I(C)_*=(x_{2}^{\alpha_{12}}x_{3}^{\alpha_{13}},x_{2}^{\alpha_{2}},x_{3}^{\alpha_{3}}-x_{2}^{\alpha_{32}}x_{4}^{\alpha_{34}},x_{4}^{\alpha_{4}},x_{3}^{\alpha_{13}}x_{4}^{\alpha_{24}})$

\item$Case \;2(b4)\;: I(C)_*=(x_{2}^{\alpha_{12}}x_{3}^{\alpha_{13}},x_{2}^{\alpha_{2}}-x_1^{\alpha_{21}}x_{4}^{\alpha_{24}} ,x_{3}^{\alpha_{3}}-x_{2}^{\alpha_{32}}x_{4}^{\alpha_{34}},x_{4}^{\alpha_{4}},x_{3}^{\alpha_{13}}x_{4}^{\alpha_{24}})$
\end{itemize}

\begin{theorem}\label{thm3.3} In $Case \;2(b1)$,  the minimal free resolution for the tangent cone of C is
\begin{center}
$ 0 \rightarrow R^2 \xrightarrow{\phi_3} R^6 \xrightarrow{\phi_2} R^5\xrightarrow{\phi_1} R^1  \rightarrow R/I(C)_* \rightarrow 0 $
\end{center}
\noindent where
{\footnotesize
$$
\phi_1=
\begin{pmatrix}
x_2^{\alpha_{12}}x_3^{\alpha_{13}} &
x_{2}^{\alpha_{2}}&
x_{3}^{\alpha_{3}} &
x_{4}^{\alpha_{4}}&
x_{3}^{\alpha_{13}}x_{4}^{\alpha_{24}} 
\end{pmatrix},
$$\\[-1cm]

$$
\phi_2=
\begin{pmatrix}
x_4^{\alpha_{24}} &  x_3^{\alpha_{43}}  & x_2^{\alpha_{32}}  & 0& 0 & 0  \\[0.5mm]
0 & 0& -x_3^{\alpha_{13}} & 0 & 0 & -x_4^{\alpha_{4}}  \\[0.5mm]
0 &- x_2^{\alpha_{12}} &  0  & 0 &- x_4^{\alpha_{24}}   & 0\\[0.5mm]
0 & 0 & 0 & -x_3^{\alpha_{13}} &   0  & x_2^{\alpha_{2}} \\[0.5mm]
-x_2^{\alpha_{12}} & 0 & 0 &   x_4^{\alpha_{34}} & x_3^{\alpha_{43}}  & 0\\
\end{pmatrix},
\hspace{3mm}
\phi_3=
\begin{pmatrix}
x_{3}^{\alpha_{43}} &  x_{2}^{\alpha_{32}}x_{4}^{\alpha_{34}}\\[0.5mm]
-x_{4}^{\alpha_{24}} & 0 \\[0.5mm]
0 & -x_{4}^{\alpha_{4}} \\[0.5mm]
0 & x_{2}^{\alpha_{2}}\\[0.5mm]
x_{2}^{\alpha_{12}}& 0 \\[0.5mm]
0 & x_{3}^{\alpha_{13}} 
\end{pmatrix},
$$
}

\noindent in $Case \;2(b2)$, the minimal free resolution for the tangent cone of C is

\begin{center}
$ 0 \rightarrow R^2 \xrightarrow{\phi_3} R^6 \xrightarrow{\phi_2} R^5\xrightarrow{\phi_1} R^1  \rightarrow R/I(C)_* \rightarrow 0 $
\end{center}

{\footnotesize
$$
\phi_1=
\begin{pmatrix}
x_2^{\alpha_{12}}x_3^{\alpha_{13}} &
x_{2}^{\alpha_{2}}-x_{1}^{\alpha_{21}}x_{4}^{\alpha_{24}}&
x_{3}^{\alpha_{3}} &
x_{4}^{\alpha_{4}}&
x_{3}^{\alpha_{13}}x_{4}^{\alpha_{24}} 
\end{pmatrix},
$$\\[-5mm]

$$
\phi_2=
\begin{pmatrix}
x_4^{\alpha_{24}} & x_3^{\alpha_{43}}  &  x_2^{\alpha_{32}} & 0 & 0 & 0 \\[0.5mm]
0 & 0 & -x_3^{\alpha_{13}} & 0 & 0 &  -x_4^{\alpha_{4}}  \\[0.5mm]
0 &  -x_2^{\alpha_{12}} & 0 & 0  & -x_{4}^{\alpha_{24}} & 0\\[0.5mm]
0 & 0 & 0 & -x_3^{\alpha_{13}} & 0  & x_2^{\alpha_{2}}-x_{1}^{\alpha_{21}}x_{4}^{\alpha_{24}} \\[0.5mm]
-x_2^{\alpha_{12}} & 0 & -x_1^{\alpha_{21}}&  x_4^{\alpha_{34}} & x_3^{\alpha_{43}} & 0\\
\end{pmatrix},
\phi_3=
\begin{pmatrix}
x_{3}^{\alpha_{43}} &  x_{2}^{\alpha_{32}}x_{4}^{\alpha_{34}} \\[0.5mm]
-x_{4}^{\alpha_{24}} & 0 \\[0.5mm]
0 & -x_{4}^{\alpha_{4}} \\[0.5mm]
0 & x_{2}^{\alpha_{2}}-x_{1}^{\alpha_{21}}x_{4}^{\alpha_{24}} \\[0.5mm]
 x_{2}^{\alpha_{12}} & 0 \\[0.5mm]
 0 & x_{3}^{\alpha_{13}}
\end{pmatrix},
$$
}

\noindent in $Case \;2(b3)$, the minimal free resolution of the tangent cone
of $C$ is
\begin{center}
$0 \rightarrow R^1 \xrightarrow{\phi_3} R^5\xrightarrow{\phi_2} R^5 \xrightarrow{\phi_1} R^1  \rightarrow R/I(C)_* \rightarrow 0$
\end{center}
\noindent where

{\footnotesize
$$
\phi_1=
\begin{pmatrix}
x_2^{\alpha_{12}}x_3^{\alpha_{13}} &
x_{2}^{\alpha_{2}}&
x_{3}^{\alpha_{3}}-x_{2}^{\alpha_{32}}x_{4}^{\alpha_{34}} &
x_{4}^{\alpha_{4}}&
x_{3}^{\alpha_{13}}x_{4}^{\alpha_{24}} 
\end{pmatrix},
$$\\[-1cm]

$$
\phi_2=
\begin{pmatrix}
-x_4^{\alpha_{24}} & 0 & 0 & -x_2^{\alpha_{32}}  &  -x_3^{\alpha_{43}}  \\[0.5mm]
0 & 0 & 0 & x_3^{\alpha_{13}}  &  x_4^{\alpha_{34}}  \\[0.5mm]
0 &  -x_4^{\alpha_{24}} & 0 & 0  & x_{2}^{\alpha_{12}} \\[0.5mm]
0  & -x_2^{\alpha_{32}} & -x_3^{\alpha_{13}} & 0 & 0 \\[0.5mm]
x_2^{\alpha_{12}} &  x_3^{\alpha_{43}}&  x_4^{\alpha_{34}} & 0 & 0\\
\end{pmatrix},
\hspace{3mm}
\phi_3=
\begin{pmatrix}
x_{3}^{\alpha_{3}}-x_{2}^{\alpha_{32}}x_{4}^{\alpha_{34}}\\[0.5mm]
-x_{2}^{\alpha_{12}}x_{3}^{\alpha_{13}} \\[0.5mm]
 x_{2}^{\alpha_{2}} \\[0.5mm]
 x_{4}^{\alpha_{4}} \\[0.5mm]
-x_{3}^{\alpha_{13}} x_{4}^{\alpha_{24}}  \\
\end{pmatrix},
$$
}

\noindent lastly, in $Case \;2(b4)$, the minimal free resolution for the tangent cone of $C$ is
\begin{center}
$0 \rightarrow R^1 \xrightarrow{\phi_3} R^5\xrightarrow{\phi_2} R^5 \xrightarrow{\phi_1} R^1  \rightarrow R/I(C)_* \rightarrow 0$
\end{center}
\noindent where

{\footnotesize
$$
\phi_1=
\begin{pmatrix}
x_{2}^{\alpha_{12}}x_{3}^{\alpha_{13}} &
x_{2}^{\alpha_{2}}-x_{1}^{\alpha_{21}}x_{4}^{\alpha_{24}}&
x_{3}^{\alpha_{3}}-x_{2}^{\alpha_{32}}x_{4}^{\alpha_{34}} &
x_{4}^{\alpha_{4}}&
x_{3}^{\alpha_{13}}x_{4}^{\alpha_{24}}
\end{pmatrix},
$$

$$
\phi_2=
\begin{pmatrix}
x_4^{\alpha_{24}} & x_2^{\alpha_{32}}  &  x_3^{\alpha_{43}} & 0 & 0 \\[0.5mm]
0 & -x_3^{\alpha_{13}}  &  -x_4^{\alpha_{34}} & 0 & 0  \\[0.5mm]
0 &  0 & -x_2^{\alpha_{12}} & 0  & -x_{4}^{\alpha_{24}} \\[0.5mm]
0  & 0 & -x_1^{\alpha_{21}} & -x_3^{\alpha_{13}} &  -x_2^{\alpha_{32}} \\[0.5mm]
-x_2^{\alpha_{12}} &  -x_1^{\alpha_{21}}& 0 &   x_4^{\alpha_{34}} & x_3^{\alpha_{43}}\\[0.5mm]
\end{pmatrix},
\hspace{5mm}
\phi_3=
\begin{pmatrix}
x_{3}^{\alpha_{3}}-x_{2}^{\alpha_{32}}x_{4}^{\alpha_{34}}\\[0.5mm]
 x_{4}^{\alpha_{4}} \\[0.5mm]
-x_{3}^{\alpha_{13}}x_{4}^{\alpha_{24}} \\[0.5mm] 
 -x_{2}^{\alpha_{2}}+x_{1}^{\alpha_{21}}x_{4}^{\alpha_{24}} \\[0.5mm]
x_{2}^{\alpha_{12}}x_{3}^{\alpha_{13}} \\
\end{pmatrix}.
$$
}
\end{theorem}

\begin{proof} $Case \;2(b1) : $
$rank(\phi_1)=1$ and $rank(\phi_3)=2$. Since every $5 \times 5$ minors of $\phi_2$ is zero, $rank(\phi_2) \leq 4$. 
In matrix $\phi_2$, deleting the 1st and the 6th columns, and the 3rd row, we have $x_3^{2\alpha_{3}}$ and  deleting the 2nd row, and the 4th and the 5th columns, we get 
$x_2^{2\alpha_{2}+{\alpha_{12}}}$  as $4 \times 4-$minors of $\phi_2$. These two determinants are relatively prime, so $I(\phi_2)$ contains a regular sequence of length 2. 
Among the 2-minors of $\phi_3$, we have $-x_2^{\alpha_{2}+\alpha_{12}}$, $x_{3}^{\alpha_{3}}$ and
$x_4^{\alpha_{4}+\alpha_{24}}$. Since these are powers of different variables, $I(\phi_3)$ contains a regular sequence of length 3.\\

$Case \;2(b2) : $
$rank(\phi_1)=1$ and $rank(\phi_3)=2$. In matrix 
$\phi_2$, deleting the 3rd row, and the 1st and the 6th columns, we have $x_3^{2{\alpha_{3}}}$ and  deleting the 2nd and the 3rd columns, and the 4th row, we obtain $-x_4^{2\alpha_{4} +\alpha_{24}}$ and these determinants are relatively prime. 
Among the 2-minors of $\phi_3$, we have $-x_2^{\alpha_{12}}{f_2}$, $x_3^{\alpha_{3}}$ and $x_4^{\alpha_{4}+\alpha_{24}}$. Since $x_2$ is a nonzero divisor modulo $\{x_3,x_4\}$, $I(\phi_3)$ contains a regular sequence of length 3.\\

$Case \;2(b3) : $
Clearly, $rank(\phi_1)=rank(\phi_3)=1$.
In matrix 
$\phi_2$, deleting the 3rd column and the 2nd row, we have $-x_{2}^{2\alpha_{2}}$ and deleting the 4th row and the 4th column, we obtain $x_{4}^{2\alpha_{4}}$ and these  are relatively prime.\\ 

$Case \;2(b4) : $
As in the above case, 
$rank(\phi_1)=rank(\phi_3)=1$.
In the matrix
$\phi_2$, deleting the 1st row and the 5th column, we have $-x_2^{2{\alpha_{12}}}x_3^{2{\alpha_{13}}}$ and  deleting the 4th row and the 2nd column, we obtain $-x_{4}^{2\alpha_{4}}$ and they  are relatively prime. 
\end{proof}

\vspace{2mm}\noindent\textbf{Case 3(a) :} In this case,
$$f_{1}=x_{1}^{\alpha_{1}}-x_{2}^{\alpha_{12}}x_{4}^{\alpha_{14}},
f_{2}=x_{2}^{\alpha_{2}}-x_{1}^{\alpha_{21}}x_{3}^{\alpha_{23}},
f_{3}=x_{3}^{\alpha_{3}}-x_{1}^{\alpha_{31}}x_{4}^{\alpha_{34}}, \\
f_{4}=x_{4}^{\alpha_{4}}-x_{2}^{\alpha_{42}}x_{3}^{\alpha_{43}}$$ and 
$$f_{5}=x_{1}^{\alpha_{31}}x_{2}^{\alpha_{42}}-x_{3}^{\alpha_{23}}x_{4}^{\alpha_{14}}$$

\noindent Here, $\alpha_{1}=\alpha_{21}+\alpha_{31}$,  $\alpha_{2}=\alpha_{12}+\alpha_{42}$,
 $\alpha_{3}=\alpha_{23}+\alpha_{43}$,  $\alpha_{4}=\alpha_{14}+\alpha_{34}.$ The condition $n_1<n_2<n_3<n_4$ gives
$\alpha_1>\alpha_{12}+\alpha_{14}$ and $ \alpha_4<\alpha_{42}+\alpha_{43}$.
The extra conditions $\alpha_2\leq \alpha_{21}+\alpha_{23}$,  $\alpha_3\leq \alpha_{31}+\alpha_{34}$
and Lemma 5.5.1 in \cite{greuel-pfister} imply that the defining ideal $I(C)_*$ of the tangent cone is generated by the following sets: \vspace{0.3cm}

\begin{itemize}
\item$Case \;3(a1)\;: I(C)_*=(x_{2}^{\alpha_{12}}x_{4}^{\alpha_{14}}, x_{2}^{\alpha_{2}}, x_{3}^{\alpha_{3}}, x_{4}^{\alpha_{4}}, 
x_{3}^{\alpha_{23}}x_{4}^{\alpha_{14}})$

\item$Case \;3(a2)\;:  I(C)_*=(x_{2}^{\alpha_{12}}x_{3}^{\alpha_{14}},x_{2}^{\alpha_{2}}-x_1^{\alpha_{21}}x_{3}^{\alpha_{23}} ,x_{3}^{\alpha_{3}},x_{4}^{\alpha_{4}},x_{3}^{\alpha_{23}}x_{4}^{\alpha_{14}})$

\item$Case \;3(a3)\;: I(C)_*=(x_{2}^{\alpha_{12}}x_{4}^{\alpha_{14}},x_{2}^{\alpha_{2}},x_{3}^{\alpha_{3}}-x_{1}^{\alpha_{31}}x_{4}^{\alpha_{34}},x_{4}^{\alpha_{4}},x_{3}^{\alpha_{23}}x_{4}^{\alpha_{14}})$

\item$Case \;3(a4)\;: I(C)_*=(x_{2}^{\alpha_{12}}x_{4}^{\alpha_{14}},x_{2}^{\alpha_{2}}-x_1^{\alpha_{21}}x_{3}^{\alpha_{23}} ,x_{3}^{\alpha_{3}}-x_{1}^{\alpha_{31}}x_{4}^{\alpha_{34}},x_{4}^{\alpha_{4}},x_{3}^{\alpha_{23}}x_{4}^{\alpha_{14}})$
\end{itemize}

\begin{theorem}\label{thm3.4} 
In $Case \;3(a1)$, then the minimal free resolution for the tangent cone of C is
\begin{center}
$ 0 \rightarrow R^2 \xrightarrow{\phi_3} R^6 \xrightarrow{\phi_2} R^5\xrightarrow{\phi_1} R^1  \rightarrow R/I(C)_* \rightarrow 0 $
\end{center}
\noindent where
{\footnotesize
$$
\phi_1=
\begin{pmatrix}
x_2^{\alpha_{12}}x_4^{\alpha_{14}} &
x_{2}^{\alpha_{2}}&
x_{3}^{\alpha_{3}} &
x_{4}^{\alpha_{4}}&
x_{3}^{\alpha_{23}}x_{4}^{\alpha_{14}} 
\end{pmatrix},
$$\\[-1cm]

$$
\phi_2=
\begin{pmatrix}
0 & 0 & x_3^{\alpha_{23}} &  x_4^{\alpha_{34}}  & x_2^{\alpha_{42}}  & 0\\[0.5mm]
0 & -x_3^{\alpha_{3}} & 0 & 0 & -x_4^{\alpha_{14}} & 0  \\[0.5mm]
x_4^{\alpha_{14}} &  x_2^{\alpha_{2}} & 0 & 0 & 0 & 0\\[0.5mm]
0 & 0 & 0 & -x_2^{\alpha_{12}} &   0  & -x_3^{\alpha_{23}} \\[0.5mm]
-x_3^{\alpha_{43}} & 0 &  -x_2^{\alpha_{12}} & 0  & 0 & x_4^{\alpha_{34}}\\
\end{pmatrix},
\hspace{1mm}
\phi_3=
\begin{pmatrix}
 x_{2}^{\alpha_{2}} & 0 \\[0.5mm]
-x_{4}^{\alpha_{14}} & 0 \\[0.5mm]
-x_{2}^{\alpha_{42}}x_{3}^{\alpha_{43}} & x_{4}^{\alpha_{34}} \\[0.5mm]
0 & -x_{3}^{\alpha_{23}}\\[0.5mm]
x_{3}^{\alpha_{3}}& 0 \\[0.5mm]
0 & x_{2}^{\alpha_{12}} 
\end{pmatrix},
$$
}

\noindent in $Case \;3(a2)$, the minimal free resolution for the tangent cone of $C$ is

\begin{center}
$ 0 \rightarrow R^2 \xrightarrow{\phi_3} R^6 \xrightarrow{\phi_2} R^5\xrightarrow{\phi_1} R^1  \rightarrow R/I(C)_* \rightarrow 0 $
\end{center}

{\footnotesize
$$
\phi_1=
\begin{pmatrix}
x_2^{\alpha_{12}}x_4^{\alpha_{14}} &
x_{2}^{\alpha_{2}}-x_{1}^{\alpha_{21}}x_{3}^{\alpha_{23}}&
x_{3}^{\alpha_{3}} &
x_{4}^{\alpha_{4}}&
x_{3}^{\alpha_{23}}x_{4}^{\alpha_{14}} 
\end{pmatrix},
$$\\[-5mm]
$$
\phi_2=
\begin{pmatrix}
0 & 0 & x_3^{\alpha_{23}} & x_4^{\alpha_{34}}  &  x_2^{\alpha_{42}} & 0  \\[0.5mm]
0 &  -x_3^{\alpha_{3}}& 0 & 0 &  -x_4^{\alpha_{14}}& 0  \\[0.5mm]
x_{4}^{\alpha_{14}} & x_2^{\alpha_{2}}-x_{1}^{\alpha_{21}}x_{3}^{\alpha_{23}} &  0 & 0  &  0 & 0\\[0.5mm]
0 & 0 & 0 & -x_2^{\alpha_{12}} & 0  &   -x_3^{\alpha_{23}}\\[0.5mm]
-x_3^{\alpha_{43}} & 0 & -x_2^{\alpha_{12}}&  0 & -x_1^{\alpha_{21}} & x_4^{\alpha_{34}}\\
\end{pmatrix},
\phi_3=
\begin{pmatrix}
x_{2}^{\alpha_{2}}-x_{1}^{\alpha_{21}}x_{3}^{\alpha_{23}} &  0\\[0.5mm]
-x_{4}^{\alpha_{14}} & 0 \\[0.5mm]
-x_{2}^{\alpha_{42}}x_{3}^{\alpha_{43}}  & x_{4}^{\alpha_{34}} \\[0.5mm]
 0 & -x_{3}^{\alpha_{23}}\\[0.5mm]
x_{3}^{\alpha_{3}} & 0 \\[0.5mm]
0 & x_{2}^{\alpha_{12}} 
\end{pmatrix},
$$
}

\noindent in $Case \;3(a3)$, the minimal free resolution of the tangent cone of $C$ is
\begin{center}
$ 0 \rightarrow R^2 \xrightarrow{\phi_3} R^6 \xrightarrow{\phi_2} R^5\xrightarrow{\phi_1} R^1  \rightarrow R/I(C)_* \rightarrow 0 $
\end{center}

\noindent where

{\footnotesize
$$
\phi_1=
\begin{pmatrix}
x_2^{\alpha_{12}}x_4^{\alpha_{14}} &
x_{2}^{\alpha_{2}}&
x_{3}^{\alpha_{3}}-x_{1}^{\alpha_{31}}x_{4}^{\alpha_{34}} &
x_{4}^{\alpha_{4}}&
x_{3}^{\alpha_{23}}x_{4}^{\alpha_{14}} 
\end{pmatrix},
$$\\[-5mm]

$$
\phi_2=
\begin{pmatrix}
0 & 0 & x_3^{\alpha_{23}} &  x_4^{\alpha_{34}}  &  x_2^{\alpha_{42}} & 0 \\[0.5mm]
0 & -x_{3}^{\alpha_{3}}+x_{1}^{\alpha_{31}}x_{4}^{\alpha_{34}} & 0 & 0   &  -x_4^{\alpha_{14}}  & 0\\[0.5mm]
x_4^{\alpha_{14}} &  x_{2}^{\alpha_{2}} & 0   & 0 & 0 & 0\\[0.5mm]
x_1^{\alpha_{31}} & 0 & 0 & -x_2^{\alpha_{12}} & 0 &  -x_3^{\alpha_{23}} \\[0.5mm]
-x_3^{\alpha_{43}} &  0 & -x_2^{\alpha_{12}}&  0 & 0 &  x_4^{\alpha_{34}}\\
\end{pmatrix},
\phi_3=
\begin{pmatrix}
x_{2}^{\alpha_{2}}  &  0\\[0.5mm]
-x_{4}^{\alpha_{14}} & 0 \\[0.5mm]
-x_{2}^{\alpha_{42}}x_{3}^{\alpha_{43}}  & x_{4}^{\alpha_{34}} \\[0.5mm]
x_{1}^{\alpha_{31}}x_{2}^{\alpha_{42}}& -x_{3}^{\alpha_{23}}\\[0.5mm]
x_{3}^{\alpha_{3}}-x_{1}^{\alpha_{31}}x_{4}^{\alpha_{34}}  & 0 \\[0.5mm]
0 & x_{2}^{\alpha_{12}} 
\end{pmatrix},
$$
}

\noindent lastly, in $Case \;3(a4)$, then the minimal free resolution of the tangent cone is
\begin{center}
$ 0 \rightarrow R^2 \xrightarrow{\phi_3} R^6 \xrightarrow{\phi_2} R^5\xrightarrow{\phi_1} R^1  \rightarrow R/I(C)_* \rightarrow 0 $
\end{center}

\noindent where

{\footnotesize
$$
\phi_1=
\begin{pmatrix}
x_{2}^{\alpha_{12}}x_{4}^{\alpha_{14}} &
x_{2}^{\alpha_{2}}-x_{1}^{\alpha_{21}}x_{3}^{\alpha_{23}}&
x_{3}^{\alpha_{3}}-x_{1}^{\alpha_{31}}x_{4}^{\alpha_{34}} &
x_{4}^{\alpha_{4}}&
x_{3}^{\alpha_{23}}x_{4}^{\alpha_{14}}
\end{pmatrix},
$$\\[-5mm]

$$
\phi_2=
\begin{pmatrix}
x_3^{\alpha_{23}} & x_2^{\alpha_{42}}  &  x_4^{\alpha_{34}} & 0 & 0  & 0\\[0.5mm]
0 & -x_4^{\alpha_{14}}  &   0 & 0  & 0 & x_{3}^{\alpha_{3}}-x_{1}^{\alpha_{31}}x_{4}^{\alpha_{34}} \\[0.5mm]
0 &  0 &  0  & -x_{4}^{\alpha_{14}} & 0 & -x_{2}^{\alpha_{2}}+x_{1}^{\alpha_{21}}x_{3}^{\alpha_{23}} \\[0.5mm]
0  & 0 & -x_2^{\alpha_{12}} & -x_1^{\alpha_{31}} &  -x_3^{\alpha_{23}} & 0\\[0.5mm]
-x_2^{\alpha_{12}} &  -x_1^{\alpha_{21}}& 0 &   x_3^{\alpha_{43}} & x_4^{\alpha_{34}} & 0\\
\end{pmatrix},
\phi_3=
\begin{pmatrix}
x_{4}^{\alpha_{34}}  &  x_{2}^{\alpha_{42}}x_{3}^{\alpha_{43}}\\[0.5mm]
 0  & -x_{3}^{\alpha_{3}}+x_{1}^{\alpha_{31}}x_{4}^{\alpha_{34}}\\[0.5mm]
-x_{3}^{\alpha_{23}}  & -x_{1}^{\alpha_{31}}x_{2}^{\alpha_{42}} \\[0.5mm]
0 & x_{2}^{\alpha_{2}}-x_{1}^{\alpha_{21}}x_{3}^{\alpha_{23}}\\[0.5mm]
x_{2}^{\alpha_{12}}  & x_{1}^{\alpha_{1}} \\[0.5mm]
0 & -x_{4}^{\alpha_{14}} 
\end{pmatrix}.
$$
}

\end{theorem}

\begin{proof} $Case \;3(a1) :$ 
Clearly, $rank(\phi_1)=1$ and  $rank(\phi_3)=2$.  
In matrix $\phi_2$, deleting the 4th and the 5th columns, and the 3rd row, we have $-x_{3}^{2\alpha_{3}+{\alpha_{23}}}$ and similarly, deleting the 2nd and the 3rd columns, and the 4th row, we obtain 
$x_4^{2\alpha_{4}}$  as $4 \times 4-$minors of $\phi_2$. These two determinants are relatively prime, so $I(\phi_2)$ contains a regular sequence of length 2. 
Among the 2-minors of $\phi_3$, we have $x_2^{\alpha_{2}+\alpha_{12}}$, $x_3^{\alpha_{3}+\alpha_{23}}$ and
$-x_4^{\alpha_{4}}$. Since they are powers of different variables, $I(\phi_3)$ contains a regular sequence of length 3.\\

$Case \;3(a2) :$ 
$rank(\phi_1)=1$ and  $rank(\phi_3)=2$. In matrix 
$\phi_2$, deleting the 3rd row, and the 4th and the 5th columns, we have $-x_3^{2{\alpha_{3}}+{\alpha_{23}}}$ and  deleting the 2nd and the 3rd columns, and the 4th row, we obtain $-x_4^{2\alpha_{4}}$ and these determinants are relatively prime. Among the 2-minors of $\phi_3$, we have 
$x_2^{\alpha_{12}}{f_2}$, $x_3^{\alpha_{3}+{\alpha_{23}}}$ and $-x_4^{\alpha_{4}}$. Since 
$x_2$ is a nonzero divisor modulo $\{x_3, x_4\}$, $I(\phi_3)$ contains a regular sequence of length 3.\\

$Case \;3(a3) :$
$rank(\phi_1)=1$ and  $rank(\phi_3)=2$. 
In matrix $\phi_2$, deleting the 1st and the 6th columns, and the 2nd row, we have $x_{2}^{{2\alpha_{2}}+{\alpha_{12}}}$ and similarly, deleting the 2nd and the 3rd columns, and the 4th row, we obtain $-x_{4}^{2\alpha_{4}}$ and these determinants are relatively prime. Among the 2-minors of $\phi_3$, we have $x_2^{\alpha_{2}+{\alpha_{12}}}$, $x_3^{\alpha_{23}}{f_3}$ and $-x_4^{\alpha_{4}}$ 
and  $x_3$ is a nonzero divisor modulo $\{x_2, x_4\}$. Thus, $I(\phi_3)$ contains a regular sequence of length 3.\\

$Case \;3(a4) :$
As in the above cases, 
$rank(\phi_1)=1$ and  $rank(\phi_3)=2$. In the matrix $\phi_2$, deleting the 1st row, and the 5th and the 6th columns, we obtain $-x_{2}^{2\alpha_{12}}x_{4}^{2{\alpha_{14}}}$ and deleting the 3rd row, and the 2nd and the 3rd columns, we have $x_{3}^{\alpha_{23}}{f_3}^{2}$ and they are relatively prime.  Among the 2-minors of $\phi_3$, we have $-x_2^{\alpha_{12}}{f_2}$, $-x_3^{\alpha_{23}}{f_3}$ and $-x_4^{\alpha_{4}}$ 
and $x_4$ is a nonzero divisor modulo $\{x_2^{\alpha_{12}}{f_2}, x_3^{\alpha_{23}}{f_3}\}$.  Thus, $I(\phi_3)$ contains a regular sequence of length 3.
\end{proof}

\section{Hilbert Function}
In \cite{arslan-mete}, Arslan and Mete showed that the Hilbert function is non-decreasing for local Gorenstein rings with embedding dimension four associated to non-complete intersection monomial curve $C$ in all above cases.
In this section, we compute the Hilbert function of the tangent cone of $C$, if $C$ is a Gorenstein non-complete intersection monomial curve in 
$\mathbb{A}^4$ as in Case 1(a). For the other aforementioned cases, one can get similar results. 
 
Theorem 3.1 implies that the tangent cone of $C$ has the following graded minimal free resolution:

\begin{center}
$0 \longrightarrow F_3 \xrightarrow {\phi_{3}} F_2 \xrightarrow {\phi_{2}} F_1 \xrightarrow {\phi_{1}} R \rightarrow  R/I(C)_* \rightarrow 0$
\end{center}

\noindent where $F_1 = \bigoplus\limits_{i=1}^{5} R(-b_i)$,  $F_2 = \bigoplus\limits_{i=1}^{6}R(-c_i)$ and $F_3 = \bigoplus\limits_{i=1}^{2}R(-d_i)$. Here,  the numbers $b_i$ are called the 1st Betti degrees, $c_i$ are called 2nd Betti degrees and $d_i$ are called 3rd Betti degrees. 
\begin{corollary} Under the hypothesis of Theorem 3.1,  Betti degrees of the minimal graded free resolution of the tangent cone of  $C$ is given by

\begin{center}
$0 \longrightarrow F_3 \xrightarrow {\phi_{3}} F_2 \xrightarrow {\phi_{2}} F_1 \xrightarrow {\phi_{1}} R \rightarrow  R/I(C)_* \rightarrow 0$
\end{center}
\noindent  where  $B_1=\{b_1,b_2,b_3,b_4,b_5\}$, $B_2=\{c_1,c_2,c_3,c_4,c_5,c_6\}$, $B_3=\{d_1,d_2\}$ and
\begin{center}
$b_1=\alpha_{13}+\alpha_{14}$, $b_2=\alpha_{2}$, $b_3=\alpha_{3}$,
$b_4=\alpha_{4}$, $b_5=\alpha_{32}+\alpha_{14}$,
\end{center}
\begin{center}
$c_1=\alpha_{3}+\alpha_{14}$, $c_2=\alpha_{2}+\alpha_{3}$, $c_3=\alpha_{4}+\alpha_{13}$,
$c_4=\alpha_{32}+\alpha_{13}+\alpha_{14}$, $c_5=\alpha_{32}+\alpha_{4}$, $c_6=\alpha_{14}+\alpha_{2}$
\end{center}
\begin{center}
$d_1=\alpha_{2}+\alpha_{3}+\alpha_{14}$, $d_2=\alpha_{4}+\alpha_{13}+\alpha_{32}$.
\end{center}

\end{corollary}

The following corollary stems from the well-known fact that
\begin{center}
$H_{G}(i)= H_{R}(i) - H_{F_1}(i) + H_{F_2}(i) - H_{F_3}(i).$

\end{center}
\begin{corollary} Under the hypothesis of Theorem 3.1., the Hilbert function of the tangent cone of $C$ is given by\\

\noindent 
\begin{eqnarray*}
H_{G}(i)\hspace{-3mm}&=&\hspace{-3mm}\left(\!\!\!\begin{array}{c} i+3 \\ 3 \end{array}\!\!\! \right)
-\left(\! \!\! \begin{array}{c} i-b_1+3 \\ 3 \end{array}\!\!\! \right)
-\left(\! \!\! \begin{array}{c} i-b_2+3 \\ 3 \end{array}\!\!\! \right)
-\left(\! \!\! \begin{array}{c} i-b_3+3 \\ 3 \end{array}\!\!\! \right)\\
&&\hspace{-5.5mm}
-\left(\! \!\! \begin{array}{c} i-b_4+3 \\ 3 \end{array}\!\!\! \right)
-\left(\! \!\! \begin{array}{c} i-b_5+3 \\ 3 \end{array}\!\!\! \right)
+\left(\! \!\! \begin{array}{c} i-c_1+3 \\ 3 \end{array}\!\!\! \right)\\
&&\hspace{-5.5mm}
+\left(\! \!\! \begin{array}{c} i-c_2+3 \\ 3 \end{array}\!\!\! \right)
+\left(\! \!\! \begin{array}{c} i-c_3+3 \\ 3 \end{array}\!\!\! \right)
+\left(\! \!\! \begin{array}{c} i-c_4+3 \\ 3 \end{array}\!\!\! \right)\\
&&\hspace{-5.5mm}
+\left(\! \!\! \begin{array}{c} i-c_5+3 \\ 3 \end{array}\!\!\! \right)
+\left(\! \!\! \begin{array}{c} i-c_6+3 \\ 3 \end{array}\!\!\! \right)
-\left(\! \!\! \begin{array}{c} i-d_1+3 \\ 3 \end{array}\!\!\! \right)\\
&&\hspace{-5.5mm}
-\left(\! \!\! \begin{array}{c} i-d_2+3 \\ 3 \end{array}\!\!\! \right),
\end{eqnarray*}

\noindent for $i \geq 0$.
\end{corollary}

\section{Funding}
Authors acknowledge partial financial support from Balikesir University under the project number BAP 2017/108.

\bibliographystyle{amsplain}

\begin{thebibliography}{10}

\bibitem{arslan-katsabekis-nalbandiyan} Arslan F,  Katsabekis A,  Nalbandiyan M.
On the Cohen-Macaulayness of tangent cones of monomial curves in $\mathbb{A}^4(K)$. Turkish Journal of Mathematics 2019; 43 : 1425-1446.

\bibitem{arslan-mete}  Arslan F, Mete P. Hilbert functions of Gorenstein monomial curves.
Proceedings of the American Mathematical Society 2007; 135 : 1993-2002.

\bibitem{barucci-froberg-sahin}  Barucci V,  Fr\"{o}berg R,  \c{S}ahin, M. On free resolutions of some semigroup rings. Journal of Pure and Applied Algebra 2014; 218 (6) : 1107-1116.


\bibitem{bresinsky}  Bresinsky H. Symmetric semigroups of integers generated by four elements. Manuscripta Mathematica 1975; 17 : 205-219.

\bibitem{buchsbaum-eisenbud} Buchsbaum D,  Eisenbud D. What makes a complex exact?. Journal of Algebra 1973 ; 323 : 259-268.


\bibitem{greuel-pfister} Greuel GM,  Pfister G.  A Singular Introduction to Commutative
Algebra : Springer-Verlag, 2002.

\bibitem{singular}  Greuel GM, Pfister G,  Sch\"{o}nemann H. (2001).
{\sc Singular} 2.0. A Computer Algebra System for Polynomial
Computations. Centre for Computer Algebra, University of
Kaiserslautern. Available  at http://www.singular.uni-kl.de.

\bibitem{herzog-rossi-valla} Herzog J, Rossi ME, Valla, G.  On the depth of the symmetric algebra. Transactions of the American Mathematical Society  1986 ; 296(2) : 577-606.

\bibitem{kunz}  Kunz E. The value-semigroup of a
one-dimensional Gorenstein ring.  Proceedings of the American Mathematical Society 1970; 
25 : 748-751.


\bibitem{oneto-strazzanti-tamone} Oneto A, Strazzanti F, Tamone G. One-dimensional Gorenstein local rings with decreasing 
Hilbert function.  Journal of  Algebra 2017 ; 489 : 92-114.


\bibitem{rossi} Rossi ME. Hilbert functions of Cohen-Macaulay local rings. 
Commutative Algebra and its Connections to Geometry.  Contemporary Mathematics AMS 2011 ;  555 : 173-200.

\bibitem{stamate}  Stamate D. Betti numbers for numerical semigroup rings.
Multigraded Algebra and Applications. Springer Proceedings in Mathematics and Statistics 2018 ; 238 : 133-157.

\bibitem{sahin-sahin1}  \c{S}ahin M.  \c{S}ahin N. 
On pseudo symmetric monomial curves.  Communications in Algebra 2018 ; 46 (6) : 2561-2573.

\bibitem{sahin-sahin}  \c{S}ahin M.  \c{S}ahin N. 
Betti numbers for certain Cohen-Macaulay tangent cones.  Bulletin of the Australian Mathematical  Society 2019 ; 99 (1) : 68-77.

\end{thebibliography}

\end{document}